\documentclass[a4paper, 11pt]{amsart}
\usepackage[utf8]{inputenc}
\usepackage[T1, T2A]{fontenc}
\usepackage{amsmath}
\usepackage{amsthm}
\usepackage{amssymb}
\usepackage[american]{babel}
\usepackage{lineno}
\usepackage{amscd}
\usepackage{subcaption}
\usepackage{wrapfig}
\usepackage{colortbl}
\usepackage[usenames, dvipsnames]{xcolor}
\usepackage{tikz}
\usepackage{tabularx}
\modulolinenumbers[5]

\usepackage[linktocpage=true, linktoc=all,hidelinks,hypertexnames=true]{hyperref}
\hypersetup{pdfstartpage=1, pdfstartview={FitH}}

\theoremstyle{definition} 

\newtheorem{thm}{Theorem}[section]      
\newtheorem{definition}[thm]{Definition}
\newtheorem{prop}[thm]{Proposition}
\newtheorem{lemma}[thm]{Lemma}
\newtheorem{corollary}[thm]{Corollary}

\newtheorem{remark}[thm]{Remark}
\newtheorem{example}[thm]{Example}

\newcommand{\NN}{{\mathbb N}}

\newcommand{\CC}{{\mathbb C}}
\newcommand{\RR}{{\mathbb R}}
\newcommand{\QQ}{{\mathbb Q}}

\begin{document}

%\begin{frontmatter}

\title{Legendrian curves in \texorpdfstring{$\CC P^3$}{CP 3}: cubics and curves on a quadric surface}

\author[N. Kalinin]{Nikita Kalinin}\thanks{National Research University Higher School of Economics, Soyuza Pechatnikov str., 16, St. Petersburg, Russian Federation. Support from the Basic Research Program of the National Research University Higher School of Economics is gratefully acknowledged. Supported in part by Young Russian Mathematics award.}

%\begin{keyword}
%Tropical geometry, legendrian curves, macaulay2
%\end{keyword}

%\end{frontmatter}

%\linenumbers

\maketitle

\begin{abstract}  We prove that the number of legendrian rational cubics in $\CC P^3$ through three generic points and a line is three; also we classify all legendrian curves on a quadric surface. Several computations are additionally verified using Macaulay2 computer algebra system.
\end{abstract}

\section{Introduction}
Inspired by Gromov-Witten invariants, one can try to count
holomorphic curves under some additional restrictions. E.g., I. Vainsencher asked to count legendrian curves passing through prescribed number of generic points or lines. His student, \'Eden Amorim \cite{contato} used localizations to count rational legendrian curves through $2d+1$ generic lines in $\CC P^3$.  Then G. Mikhalkin proposed to me this problem as a potential topic for my thesis, however not much has been accomplished. In this paper, we show that the family of legendrian cubics passing through three generic point in $\CC P^3$ forms a line in the space of coefficients and classify all algebraic legendrian curves on a quadric surface. Some computations are performed in Macaulay2, \cite{M2}.

%The famous twistor construction obtains all harmonic surfaces in $S^4$ as projections of complex legendrian curves in $\CC P^3$.

The recent study of the complex legendrian curves %(see also \cite{Alarcon:2014mz}), by itself, 
is
motivated by minimal surfaces in four dimensional sphere. %(read \cite{meeks2012survey} and \cite{brendle2013minimal} about minimal surfaces in $\RR^3$ and $S^3$, respectively). 
The map
$(z_1,z_2,z_3,z_4)\to(z_1+jz_2,z_3+jz_4)$ from $\CC^4$ to $\mathbb H^2$
yields so-called twistor (or Penrose) map $\phi:\CC P^3\to\mathbb H P^1=S^4$, and Bryant has
shown \cite{bryant} that the images  of the legendrian
curves in $\CC P^3$ under $\phi$ are superminimal surfaces in
$S^4$. Furthermore, each minimal immersion $S^2\to S^4$ can be
obtained as $\phi(C)$ where $C$ is a rational legendrian curve in $\CC
P^3$.  Then, each Riemann surface $M$ can be mapped to a legendrian curve in $\CC P^3$,
using two meromorphic functions $(f,g)$ on $M$. This leads to the fact
that for each Riemann surface $M^2$ there exists a conformal minimal
immersion $M^2\to S^4$, and such immersions are nowadays constructed mostly by this
approach. See a recent survey \cite{alarcon2017new} about the minimal surfaces. %(see also \cite{vlachos2012isometric} about isometric deformation of minimal surfaces in $S^4$). 

The area of the image of a harmonic map $f:S^2\to S^4$ is equal to $4\pi d$ if
$f(S^2)$ comes as the projection a legendrian
rational curve in $\CC P^3$ of degree $d$. The dimension of the space
$\mathfrak M_{d,0}$ of legendrian maps $\CC P^1\to \CC P^3$ of
degree $d$ is proven to be  $2d+4$, see \cite{harmonic,0686.58009,verdier1984applications,verdier1988applications}; see
\cite{0914.58005} for the legendrian maps $\CC P^1\to\CC
P^{2n+1}$. This is done via studying the pairs $(f,g)$ of
meromorphic functions %$\big($secretly $(f,g)=(\frac{z_1}{z_2},\frac{z_3}{z_4})$$\big)$ 
of degree $d$ with the same ramification divisor.
Up to degree six the space $\mathfrak M_{d,0}$ is a smooth complex manifold, see \cite{1255.58004}.
%See also a survey \cite{contactSurvey} and references therein about minimal immersions $S^2\to S^{2n}$.

If $d\geq
g+3$, then the part of the space $\mathfrak M_{d,g}$, which consists of smooth contact curves in $\CC P^3$ of degree $d$ and
genus $g$, is smooth, \cite{complexContact,MR1608659}. Besides, a complete
intersection cannot be a contact curve \cite{buczynski2006legendrian}. That complicates the study of
the contact curves of higher genus, which was approached in \cite{0997.53043, 0902.53043}. The
dimension of $\mathfrak M_{d,g}$ is $2d-g+4$
for $d\geq \max(2g,g+2)$, \cite{0943.58007}; 
the dimension of each irreducible
component of $\mathfrak M_{d,g}$ is between $2d-4g+4$ and $2d-g+4$,
where upper bound is always attained by the totally geodesic
immersions (whose image belongs to a line) and the lower bound is obtained
on $\mathfrak M_{6,1}$ and $\mathfrak M_{8g+1+3k,g}$,\cite{0997.53043}. See
\cite{0902.53043}, for further details about other possible pairs
$(d,g)$ with non-trivial contact curve. All this means that for $g\geq 1$ we need to take
the degree $d$ of the curve at least $6$ what is now beyond our abilities
to compute with formulae even using computer.

%\q{read compactification
%  and one of this things for a coordinate change, check which
%  hyperbolic structure is there}
%\q{probably see on compactified cubics \cite{cubics} and \cite{cubicsComp}}

For a general overview of complex contact varieties and deformations
of contact curves see \cite{buczinsky}, \cite{complexContact, 0801.14014}. Real algebraic contact
structures are numerous, the questions about polynomial distributions
went back to \cite{MR0016748,MR0217061}, see Example~\ref{ex_realform}.

For the works of the same spirit we mention the study of legendrian
curves of minimal degree through two points with prescribed tangency
\cite{giluch2005real} and contact curves in $PSL(2,\CC)$ \cite{contactsl2}.

%\q{compute tropical curves for } \cite{contactsl2}.
%algebraic foliations [7] J.P. Jouanolou. Equations de Pfaff algebriques.  studied by  Vainsenher \cite{foliations}

\section{The contact structure on \texorpdfstring{$\CC P^3$}{CP^3}}

\begin{definition}
A section $\omega$ of the projectivization $P(\Omega^1(\CC P^3))$ of the cotangent bundle of $\CC P^3$ is said to be a {\it contact holomorphic form} on $\CC P^3$ if $\omega\wedge d\omega$ is nowhere zero. 

Formally, there are charts $A_i$, holomorphic $1$-forms $\omega_i$ on $A_i$, a set $f_{ij}$ of transition functions on $A_i\cap A_j$, $f_{ij}\omega_i=\omega_j$, such that  $\bigcup A_i=\CC P^3$ and  $\omega_i\wedge d\omega_i\ne 0$ on $A_i$. 

\end{definition}
Note that if $\omega$ is locally a contact form, $f$ is a function, then $f\omega$ is also a contact form since
\begin{equation}
\label{eq_trans}
f\omega\wedge d(f\omega) = f^2\omega\wedge d\omega.
\end{equation}

\begin{example}
The form $\omega =ydx-xdy+wdz-zdw$ is contact. 
\end{example}
Indeed, consider the restriction of $\omega$ to the chart $w=1$. We have $$\omega|_{w=1}=dz+ydx-xdy,$$
 $$\omega|_{w=1}\wedge d\omega|_{w=1} = -2dx\wedge dy\wedge dz\ne 0,$$ similar formulae hold in other charts.

\begin{thm}[\cite{kobayashi1959remarks}]
Each contact holomorphic form $\omega$ on $\CC P^3$ is of the type
\begin{equation}
(py-qz+aw)dx+(-px+rz+bw)dy+(qx-ry+cw)dz+(-ax-by-cz)dw
\label{eq_contactform}
\end{equation}
 where
  $a,b,c,p,q,r$ are constants and $pc+qb+ra\ne 0$. Furthermore, all
  such forms are equivalent under the $GL(4,\CC)$ action.
\end{thm}
\begin{proof}We only sketch a proof from \cite{kobayashi1959remarks}. Let $\alpha$ be a holomorphic contact form
  in $\CC P^3$. Note that the set $f_{ij}$ of transition functions defines a linear bundle whose first Chern class we denote by $c_1(\alpha)$. The form $\alpha\wedge d\alpha$ gives a section of the canonical bundle. Considering transition function \eqref{eq_trans}
we conclude that $c_1(\CC P^3)=2c_1(\alpha)$.  It means that if
$Pdx+Qdy+Rdz$ is a contact form in the chart $w=1$, then it extends to
the whole $\CC P^3$ only if the transition function to another charts have $w$ in
denominator in degree at most two. Therefore $P,Q,R$ are polynomials of
degree one. The explicit form of all such polynomials follows from a direct computation. 
\end{proof}
Quite the contrary, there are many algebraic contact structures on $\RR
P^3$.
\begin{example}
\label{ex_realform}
The following forms are contact forms on $\RR P^3$: 
$$\omega_1' =(yz^2+yw^2)dx+(-xz^2-xw^2)dy+(x^2w+y^2w+w)dz+(-x^2z-y^2z-z)dw, $$
\begin{align*}
\omega_2'=&
(x^2y+y^3+yz^2+yw^2)dx-(x^3+xy^2+xz^2+xw^2)dy\\
&+(x^2w+y^2w+z^2w+w^3)dz-(x^2z+y^2z+z^3+zw^2)dw.
\end{align*}
Note also that a small perturbation of the coefficients of a real
contact form doesn't affect the fact that $\omega'\wedge d\omega'$ never vanishes.
\end{example}

It seems not easy to enumerate real algebraic
curves which are contact with respect to these contact structures.

%\begin{problem}
%Does there exist a real algebraic contact form with all the coefficients of degree two?
%\end{problem}

\begin{prop}
Any irreducible algebraic curve $C\in \CC P^3$ which is not a collection of lines is legendrian with at most one holomorphic contact structure.
\end{prop}
Indeed, when we intersect the distribution given by \eqref{eq_contactform} with the distribution given by $\omega=yxd-xdy+wdz-zdw$, we obtain a vector field $v$ almost everywhere (except finite collection of points as the Macaulay2 code below shows). On the other hand, we know that there is a line, tangent to the obtained distribution, through each point in $\CC P^3$. Therefore the integral curves for $v$ are lines almost everywhere. Hence, the only locus where a curve, tangent to both distribution, can leave, is the set where these two contact forms coincide, i.e. a finite collection of points.

The following code in Macaulay2, \cite{M2}, obtains the ideal $J$ of the variety of the points where two contact structures coincide. Comments are starting with ``--''. What follows after ``='' is the output of the corresponding command. We use these conventions throughout this paper.
\begin{verbatim}
use QQ[p,q,r,a,b,c,x,y,z,w]
a1=p*y-q*z+a*w
a2=-p*x+r*z+b*w
a3=q*x-r*y+c*w
a4=-a*x-b*y-c*z

I=ideal(a1*x-a2*y, a2*w-a3*x,a3*z-a4*w)
C=minimalPrimes I
J=C_8  -- all the other ideals C_0,C_1,... give lines if we fix a,b,c,p,q,r
dim J  -- =7
-- 7-6=1, because we have 6 parameters p,q,r,a,b,c 
-- so it is just several lines through the origin 
-- that is, a collection of points after the homogenisation.
\end{verbatim}

The global Reeb vector field for the contact structure $\omega=ydx-xdy+wdz-zdw$ is given by $y\frac{\partial}{\partial x}
-x\frac{\partial}{\partial y}+w\frac{\partial}{\partial z} -
z\frac{\partial}{\partial w}$. Its trajectories (which are also the fibers of the Penrose map $\CC P^3\to S^4$) are given by
\begin{equation}
\varphi(t)=\left(A\frac{(t^2-1)}{(t^2+1)},2A\frac{t}{(t^2+1)},B\frac{((t+k)^2-1)}{((t+k)^2+1)},2B\frac{(t+k)}{((t+k)^2+1)}\right)
\end{equation}
and $(\frac{t^2-1}{t^2+1})'=\frac{4t}{t^2+1},(\frac{2t}{t^2+1})'=\frac{t^2-1}{t^2+1}$. So, the Reeb vector field just rotates in $xy$ plane and $zw$ plane on the same
angle. For each fixed angle this gives a linear transformation.

\section{Contact form automorphisms}
It is known that the group of automorphisms of $\CC P^3$ which
preserve the form $\omega = ydx-xdy+wdz-zdw$ is the symplectic group
$Sp(4,\CC)$. Indeed, we have 6 conditions
  on the coefficients of a matrix $A\in GL(4,\CC)$, since $A$
  preserves $\omega$,  and the condition
  $\det A\ne 0$, but one
can check (by Macaulay2 for example), that the set of such $A\in
\CC^{16}$ is a quasiprojective
variety of dimension 10. The dimension count gives $\dim Sp(4,\CC)=10$ and $\dim PGL(4,\CC)=15$, what
agrees with the fact the set of all contact structures in \eqref{eq_contactform} is five-dimensional.

\begin{prop}
\label{prop_generators}
We list the set of generators of this group $Sp(4,\CC)$.%% and its action on other contact form written as \eqref{eq_contactform} in terms of $(p,q,r,a,b,c)$):

\begin{itemize}
\begin{minipage}[b]{0.45\textwidth}
\item 1) $x\to x+\lambda y, 
\begin{pmatrix}
1 & \lambda & 0 & 0 \\
0 & 1 & 0 & 0 \\
0 & 0 & 1 & 0 \\
0 & 0 & 0 & 1 \\
\end{pmatrix}$

%\begin{cases} r\to r-\lambda q,\\ b\to b+\lambda
  %a\\ \end{cases}
\item 2) $x\to y, y\to -x, \begin{pmatrix}
0 & 1 & 0 & 0 \\
-1 & 0 & 0 & 0 \\
0 & 0 & 1 & 0 \\
0 & 0 & 0 & 1 \\
\end{pmatrix}$
\end{minipage}
\begin{minipage}[b]{0.45\textwidth}
  %\begin{cases} q\to-r,\\ r\to q,\\  a\to b,\\ b\to
%  -a\\ \end{cases}
\item 3) $x\to z, y\to w, 
\begin{pmatrix}
0 & 0 & 1 & 0 \\
0 & 0 & 0 & 1 \\
1 & 0 & 0 & 0 \\
0 & 1 & 0 & 0 \\
\end{pmatrix}$
%\begin{cases} p\to c,\\  a\to -r,\\ q\to -q,\\
%  b\to -b,\\ r\to -a,\\ c\to p\\ \end{cases} 
\item 4)  $x\to x+\lambda w, z\to z+\lambda y, 
\begin{pmatrix}
1 & 0 & 0 & \lambda \\
0 & 1 & 0 & 0 \\
0 & \lambda & 1 & 0 \\
0 & 0 & 0 & 1 \\
\end{pmatrix}
$
 %\begin{cases} q\to q,\\ p\to p-\lambda q,\\ a\to
%a,\\ r\to r,\\ b\to b-p\lambda +c\lambda +q\lambda^2,\\ c\to
%c+q\lambda\\ \end{cases}
\end{minipage}
\item 5) $x\to\lambda x,y\to y/\lambda, 
\begin{pmatrix}
\lambda & 0 & 0 & 0 \\
0 & 1/\lambda & 0 & 0 \\
0 & 0 & 1 & 0 \\
0 & 0 & 0 & 1 \\
\end{pmatrix}$
\end{itemize}
\end{prop}

%In affine coordinates $(x,y,z)$ these generators are written as follows.
%$$x\to x+cy$$
%$$z\to z+c$$
%$$x\to y, y\to -x$$
%$$x\to x/z, y\to y/z, z\to -1/z$$
%$$x\to z/y, y\to 1/y, z\to x/y$$
%$$x\to x-c, z\to z-cy$$

%Then we can bring any form to $c=0,q=1,r=a=0$. Because there are two
%invariants $p+c$ and $pc+qb+ra$.

%that is all  $x\to a_1x+b_1y+c_1z+d_1w,y\to a_2x+...$ such that $c_2a_1-c_1a_2+c_4a_3-c_3a_4=c_2b_1-c_1b_2+c_4b_3-c_3b_4=d_2a_1-d_1a_2+d_4a_3-d_3a_4=d_2b_1-d_1b_2+d_4b_3-d_3b_4=b_2a_1-b_1a_2+b_4a_3-b_3a_4=c_2d_1-c_1d_2+c_4d_3-c_3d_4=0$
%is satisfied.\ref{code_forms}

%Finally,

\begin{prop}
The restriction of a contact structure \eqref{eq_contactform} on a plane
$z=w=0$ is $p(ydx-xdy)=0$ by an easy computations. 
\end{prop}
Therefore the vector field, generated by the contact form, at a point $(x,y)$
equals to a vector $(x,y)$, so the only integral curves are the lines passing through
the origin. Since all the planes are equivalent under the action of $GL(4,\CC)$, all the planar contact curves are collections of lines.

Let us choose an arbitrary plane $L$.
\begin{prop}
Each contact curve in $L$ is a collection of lines through a point
$p\in L$. Moreover, $L$ is the contact plane at $p$, i.e. $L$ is the
zero set of $\omega$ computed at this point $p$.
\end{prop}

%\subsection{Macaulay2 code for the action of contactomorphism group on triplets of points}
%\label{sec_contactomorphisms}
\begin{prop}
All the elements of $Sp(4,\CC)$ which preserve $(0,0,0,1),(1,1,1,1),(-1,1,-1,1)$ are of the form
\begin{equation}
\mathrm{Stab^3_\mu}:x\to x, y \to y+\mu(z-x), z\to z, w\to w-\mu(z-x),
\begin{pmatrix}
1 & 0 & 0 & 0 \\
-\mu & 1 & \mu & 0 \\
0 & 0 & 1 & 0 \\
\mu & 0 & \mu & 1 \\
\end{pmatrix}
\label{eq_fixator}
\end{equation}
\end{prop}

\begin{proof}
Direct computation.
\end{proof}

It is easy to send any point of $\CC P^3$ to $(0,0,0,1)$ by an
element of $\mathrm{Sp}(4)$. Then, points of $\CC P^3$ can be divided in
two classes: those, lying on the plane $L$ through $(0,0,0,1)$ such that
$\omega((0,0,0,1))|_{L}=0$ and all the others. The subgroup of
$\mathrm{Sp}(4)$, stabilizing $(0,0,0,1)$ acts on both these classes transitively. 
Now, consider a point $p$ which is not on the contact planes through
$(0,0,0,1)$ and $(1,1,1,1)$. It can be proven by a direct computation that there exists an element in the
subgroup of $\mathrm{Sp}(4)$ stabilizing $(0,0,0,1)$ and $(1,1,1,1)$ that
sends $p$ to $(-1,1,-1,1)$, thus we have the following lemma.

\begin{lemma}
\label{lem_transitive}
The group $\mathrm{Sp}(4)$ is {\it generically 3-transitive}, i.e. every three points  $p_1,p_2,p_3\in \CC P^3$ in general position can be sent to every three points $q_1,q_2,q_3\in \CC P^3$ in general position by an element $a\in \mathrm{Sp}(4,\CC)$. In general, the set $\{a\in \mathrm{Sp}(4, \CC)|a(p_i)=q_i,i=1,2,3\}$ is of dimension one.
\end{lemma}

%The following code in Macaulay2 produces a matrix preserving the contact
%form $\omega$, which brings generic points
%(a1,b1,c1,1),(a2,b2,c2,1),(a3,b3,c3,1) to (0,0,0,1),(1,1,1,1) and
%(-1,1,-1,1). The group of contactomorphisms is generated by the
%matrices A,B,C,D (see Proposition~\ref{prop_generators} and below formulae) and by straightforward combination of
%them we arrive to the answer.

\section{Curves on a hypersurface of degree two}
\label{sec_quadric}
Consider a contact form $\omega$ as in \eqref{eq_contactform}.
We will find the restriction of $\omega$ on the surface $X$, given by 
\begin{equation}
\label{eq_quadric}
\{xy-zw=0\}=Im(f\colon\CC P^1\times \CC P^1\to \CC P^3), f\colon(\mu:\mu'),(\nu:\nu')\to
(\mu\nu',\mu'\nu,\mu\nu,\mu'\nu').
\end{equation}

Note, that any irreducible hypersurface $X'$ of degree 2 in $\CC P^3$ is projectively equivalent to $X$, therefore in this way we will describe all the legendrian curves on all the non-degenerate hypersurfaces $X'$ of degree $2$.

Computing in the affine chart $(\mu,\nu)\to (\mu,\nu,\mu\nu,1)\in X$, we obtain

$$f_*:\frac{\partial}{\partial \mu}\to \frac{\partial}{\partial x} +y \frac{\partial}{\partial z},
\frac{\partial}{\partial \nu}\to \frac{\partial}{\partial y} +x
\frac{\partial}{\partial z}. $$

The fact that $\omega (f_*(M\frac{\partial}{\partial
  \mu}+N\frac{\partial}{\partial \nu})=0$ at $ (\mu,\nu,\mu\nu,1)$ is equivalent
to

$$(p\nu-q\mu\nu+a)M+(-p\mu+r\mu\nu+b)N+(q\mu-r\nu+c)(M\nu+N\mu)=0, \text{\ i.e.}$$

$$M(p\nu +a-r\nu^2+c\nu)+N(-p\mu+b+q\mu^2+c\mu)=0.$$

%So, if on $\CC^2$ one have a curve with equation $F(\mu,\nu)=0$  which is
%tangent to $\omega$ after applying $f$, then $F_\mu=(p+c)\nu +a-r\nu^2,
%F_\nu=(c-p)\mu+b+q\mu^2$.

If a curve is locally of type $(\mu(t),\nu(t))$, then its tangent vector is
given by the formula
$\mu'\frac{\partial}{\partial \mu}+\nu'\frac{\partial}{\partial\nu}$. But this, after reparametrization, rewrites as 
\begin{equation}\label{diffur}
\frac{d\mu}{dt} = (c-p)\mu+b+q\mu^2,
\frac{d\nu}{dt} = -((p+c)\nu +a-r\nu^2).
\end{equation}

%Hence now we should study the foliation induced bu this differential
%equation, $d\mu((p+c)\nu +a-r\nu^2)+d\nu((c-p)\mu+b+q\mu^2) =0$.
We are looking for the algebraic leafs of this foliation. See \cite{zbMATH05995244,zbMATH01602924} for details about space of foliations
with algebraic leafs, \cite{zbMATH05555269} for the classification of
the quadratic systems with the first integral.

\begin{example}Consider the curve $(t,t^2,t^3,1)$ which lies on the
hypersurface $\{xy-zw=0\}$. It is legendrian with respect to the form
$3dx-3dy+wdz-zdw=0$, so we put $p=3,c=1,q=a=r=b=0$ and
\eqref{diffur} becomes $(\mu',\nu') = (-2\mu,-4\nu)=(\mu,2\nu)$, hence $\mu =
e^t,\nu=e^{2t}$ which is the same as $(\mu,\nu)=(t,t^2)$, and
subsequently $(\mu,\nu,\mu\nu,1)=(t,t^2,t^3,1)$.
\end{example}

Depending on the coefficients, each equation
$\frac{dx}{dt}=c_0+c_1x+c_2x^2$ after a linear change of the coordinates (over complex numbers) becomes one in the following list: 
\begin{itemize}
\item $\frac{dx}{dt}=c,$
\item $\frac{dx}{dt}=cx,$
\item $\frac{dx}{dt}=cx^2,$
\item $\frac{dx}{dt}=c(x^2-1)$.
\end{itemize}

\begin{example}
If $\frac{d\mu}{dt}=c_0(\mu^2-1),\frac{d\nu}{dt}=c_1(\nu^2-1)$, then
$\frac{d\mu}{\mu^2-1}=c_3\frac{d\nu}{\nu^2-1}$. That implies
$$\log(\frac{\mu-1}{\mu+1})=c_4\log(\frac{\nu-1}{\nu+1})+c_5,$$ and finally
$c_6(\frac{\nu-1}{\nu+1})=c_7(\frac{\mu-1}{\mu+1})^{d_1}$ which is
algebraic if $d_1\in\QQ$.
\end{example}

To the contrary, the case $\frac{d\mu}{dt}=c_0(\mu^2-1),
\frac{d\nu}{dt}=c_1\nu^2$ always gives a non-algebraic curve if
$c_0c_1\ne 0$ because this gives an equation of the type $\frac{\mu-1}{\mu+1}=e^{1/\nu}$. So, by a direct computation we prove the following theorem.

\begin{thm}
\label{th_quadricsurface}
After a linear change of coordinates $\tilde\mu=c_0+c_1\mu,\tilde\nu=c_2+c_3\nu$
any legendrian curve on the quadric $xy-zw=0$ with parametrization \eqref{eq_quadric} can be written in
one of the following standard forms :
\begin{itemize}

\item $c_0(\frac{\nu-1}{\nu+1})^{d_1}=c_1(\frac{\mu-1}{\mu+1})^{d_2}$,
\item $c_0\nu^{d_1}=c_1\mu^{d_2}$,
\item $c_0(\frac{\nu-1}{\nu+1})^{d_1}=c_1\mu^{d_2}$ ,
\item $c_0(\frac{\mu-1}{\mu+1})^{d_1}=c_1\nu^{d_2}$,
\item $c_0\mu\nu+c_1\mu+c_2\nu=0$,
\item $c_0\mu+c_1\nu+c_2=0$,
\item $c_0\mu\nu+c_1\mu+c_2=0$, 
\item $c_0\mu\nu+c_1\nu+c_2=0$,
\item $\mu=c_0$, 
\item $\nu=c_0$,
\end{itemize}
where $c_i\in\CC,d_i\in \NN_0$ are some constants.
\end{thm}

\begin{remark}
\label{pr_quadrics}
Given this classification one might count the legendrian curves of given degree and genus lying in a quadric. For example, all rational quartics lie on a quadric surface. 
\end{remark}

\section{Legendrian curves of degrees one and two}

\begin{definition}
A map $f:M\to \CC P^3$ is {\it totally geodesic} if $f(M)$ is a legendrian line.
\end{definition}

Let us study the rational legendrian curves of degrees one and two.
In the case $\deg x,y,z=1$ or $2$, it happens that such curves are parametrized by $(f,p + qf, r + pf)$, where $f$ is a polynomial of degree 1 or 2. 

Consider a general line $l = (a_0+b_0s,a_1+b_1s,a_2+b_2s,a_3+b_3s)$ in
$\CC P^3$. Putting it into the contact form we conclude that the line $l$ is legendrian iff $a_1b_0-a_0b_1+a_3b_2-a_2b_3=0$.  This means that for a point $A$ we have one-dimensional family of
legendrian lines through $A$, this family is just the contact plane through $A$. Therefore the number of legendrian lines through one point and one line
equals one. 

%\q{through three lines?}

Let us observe one important property of legendrian lines. One can
think about a line $l$ in $\CC P^3$ as four section $x,y,z,w$ of $\mathcal O(1)$
on $\CC P^1$. Let $X,Y,Z,W$ be the roots of $x,y,z,w$, $x=a_0+b_0s, X=-\frac{a_0}{b_0}$, $y=a_1+b_1s, Y = -\frac{a_1}{b_1}$, etc.

\begin{prop} The following three conditions are equivalent:
\label{prop_balance}
\begin{itemize}
\item the line $l$ is legendrian,\\ 
\item $y(X)/z(X) = w(Z)/x(Z)$,\\ 
\item $x(Y)/w(Y)=z(W)/y(W)$.\\
\end{itemize}
\end{prop}
\begin{proof} Look at table with values of $x,y,z,w$ in $X,Y,Z,W$.

$$
\begin{pmatrix}
X=&(0,&\frac{a_1b_0-b_1a_0}{b_0},&\frac{a_2b_0-b_2a_0}{b_0},&\frac{a_3b_0-b_3a_0}{b_0})\\
Y=&(\frac{a_0b_1-a_1b_0}{b_1}, &0,&\frac{a_2b_1-a_1b_2}{b_1},&\frac{a_3b_1-a_1b_3}{b_1})\\
Z=&(\frac{a_0b_2-a_2b_0}{b_2},&\frac{a_1b_2-a_2b_1}{b_2},&0,&\frac{a_3b_2-a_2b_3}{b_2})\\
W=&(\frac{a_0b_3-a_3b_0}{b_3},&\frac{a_1b_3-a_3b_1}{b_3},&\frac{a_2b_3-a_3b_2}{b_3},&0)\\
\end{pmatrix}
$$
\end{proof}

\begin{remark}
\label{pr_weil}
Is it possible to generalize this proposition for the curves of higher
degree? %As we see in Example~\ref{ex_tropicalcubic4},  the straightforward
%generalization does not take place.
\end{remark}

\section{Legendrian cubics}
\label{sec_cubicsproof}
Let us find all the legendrian cubics passing through three generic points in $\CC P^3$. We parametrize our curve and suppose that it passes through chosen points at $t=-1,0,1$.
This imposes constraints on the coefficients of this parametrization and we will find that the corresponding subvariety of the space of coefficients of cubics through three generic points. This subvariety happens to be of dimension one (as expected) and of degree one (it was not expected). First we do it using Macaulay2 and then by hands.

\begin{verbatim}
clearAll
--coefficients of the parametrization of the cubic
mainvar=(a0,a1,a2,a3,b0,b1,b2,b3,c0,c1,c2,c3,d0,d1,d2,d3)
R=QQ[mainvar]; P=R[s];
--polynomials for each coordinate
x=a0+a1*s+a2*s*s+a3*s*s*s; y=b0+b1*s+b2*s*s+b3*s*s*s;
z=c0+c1*s+c2*s*s+c3*s*s*s; t=d0+d1*s+d2*s*s+d3*s*s*s;

ourconditions = y*diff(s,x)-x*diff(s,y)+t*diff(s,z)-z*diff(s,t)
--in M we have our relation for variables since in in the variable ourconditions 
--(as a polynomial in z) all the coef. should be zeroes
(C,M) = coefficients ourconditions
(A,B,C)=(0,1,-1) 

xA=sub(x,{s=>A}); xB=sub(x,{s=>B}); xC=sub(x,{s=>C});

yA=sub(y,{s=>A}); yB=sub(y,{s=>B}); yC=sub(y,{s=>C});

zA=sub(z,{s=>A}); zB=sub(z,{s=>B}); zC=sub(z,{s=>C});

tA=sub(t,{s=>A}); tB=sub(t,{s=>B}); tC=sub(t,{s=>C});

--choose random points
(p11,p12,p13,p14)=(29,-6,13,11)
(p21,p22,p23,p24)=(-3,-17,7,-5)
(p31,p32,p33,p34)=(16,-5,6,23)
--conditions that our curve passes through chosen points
(i1,i2,i3)=(p14*xA-p11*tA,p14*yA-p12*tA,p14*zA-p13*tA)
(j1,j2,j3)=(p24*xB-p21*tB,p24*yB-p22*tB,p24*zB-p23*tB)
(k1,k2,k3)=(p34*xC-p31*tC,p34*yC-p32*tC,p34*zC-p33*tC)

use R; N= M_0; l=i->lift(i,R);

J  = ideal(i1,i2,i3,l(N_0),l(N_1),l(N_2),l(N_3),l(N_4))
S = minimalPrimes J
J0 = S_0; J1 = S_1; J2 = S_2;
--S_3 does not exist

di=i->dim variety i; use P;
Null = ideal(x,y,z,t) --if Null is a subset of our ideal, 
-- it means that x,y,z,t are all zeroes at some point, 
-- so we are not interested in such coefficients a0,a1, ...

di J0  --=7 
di J1  --=8 that raises our suspicions that it contains Null...
di J2  --=7

--ideal(s-A) means evaluation at A
isSubset(Null, promote(J0,P)+ideal(s-A))  --=false, 
isSubset(Null, promote(J1,P)+ideal(s-A))  --=true, eliminate from our consideration!
isSubset(Null, promote(J2,P)+ideal(s-A))  --=false

use R; S0 = minimalPrimes (J0+ideal(j1,j2,j3));
J00=S0_0; J01=S0_1;  --S0_2 do not exist

use P
isSubset(Null, promote(J00,P)+ideal(s-B))  --=false 
isSubset(Null, promote(J01,P)+ideal(s-B))  --=true, eliminate!

use R; S01 = minimalPrimes (J00+ideal(k1,k2,k3))
J000=S01_0; J001=S01_1;

use P; isSubset(Null, promote(J000,P)+ideal(s-C))  --=false 
isSubset(Null, promote(J001,P)+ideal(s-C))  --=true, eliminate!

di J000  --=1
degree J000  --=1, so it is linear!
---------
S2 = minimalPrimes (J2 + ideal(j1,j2,j3))
J20=S2_0  --S2_1 does not exist

isSubset(Null, promote(J20,P)+ideal(s-B))  --=true, eliminate!
\end{verbatim}

Any rational non-planar cubic is equivalent to $(t,t^2,t^3,1)$. We can choose a contact form $\omega_1$ such that $(t,t^2,t^3,1)$ was legendrian with respect to it. 

\begin{lemma} The cubic $(t,t^2,t^3,1)$ is legendrian with respect to
only one contact structure $\omega_1=3ydx-3xdy+wdz-zdw$.
\end{lemma}

\begin{proof}
Direct calculation, using \eqref{eq_contactform}.
\end{proof}

We fix the contact form $w_1$, then by a contactomorphism we can bring any three generic points to the points $(0,0,0,1),(1,1,1,1),(-1,1,-1,1)$. The main result of this section is the following theorem (above we have just predicted that the family of such curves is a line in the space of the coefficients).

\begin{thm}
\label{th_contact}
All the rational cubics passing through $(0,0,0,1),(1,1,1,1),(-1,1,-1,1)$ and tangent to $\omega_1=3ydx-3xdy+wdz-zdw$ are of the form 
\begin{equation}
\label{eq_contactfamily}
l(t,\mu)=(t,t^2+\mu(t-t^3),t^3, 1-3\mu(t-t^3)).
\end{equation} 
\end{thm}

The result of the theorem is not surprising. This is the orbit of the action of $\mathrm{Stab^3_\mu}$ (see Eq.~\eqref{eq_fixator}) on $(t,t^2,t^3,1)$. Therefore, the only problem is to show that there are no other solutions.

\begin{corollary}
\label{cor_answer}
For each holomorphic contact form on $\CC P^3$ the number of rational contact cubics through three generic points and a line in general position is equal to three. 
\end{corollary}

\begin{proof}
We intersect the family \eqref{eq_contactfamily} with a generic line $L$ of the type $(t', p_1 + q_1 t', p_2 +
q_2t', p_3+q_3t')$. Because of the genericity, $L$ does not pass through  $(0,0,0,1)=l(0,\mu)$, therefore we may suppose that at any intersection of $L$ and $l(t,\mu)$ we have $t\ne 0$. Therefore, at a point of intersection we have $t'=ct$ for some $c$, and then 
\begin{equation}
\label{eq_lineintersection}
p_1 + q_1 t'=c(t^2+\mu(t-t^3)),p_2 +
q_2t'=ct^3,p_3+q_3t'=c(1-3\mu(t-t^3)).
\end{equation}

We have $3(p_1+q_1t')+p_3+q_3t'=c(3t^2+1)$, therefore, substituting $t'=ct$ we obtain $$c=\frac{3p_1+p_3}{3t^2-3q_1t-q_3t+1}.$$ Then, using the first equality in \eqref{eq_lineintersection}, we get $\mu = \frac{p_1+q_1ct-ct^2}{c(t-t^3)}$. Then, since $c(t^3-q_2t)=p_2$, we have $$t^3-q_2t=\frac{p_2}{3p_1+p_3}(3t^2-3q_1t-q_3t+1).$$ Choosing $p_2,q_2$ appropriately, we see that the last equation usually has three roots.
\end{proof}

\begin{corollary}
For the contact form $\omega =ydx-xdy+wdz-zdw$, the parametrization of the family of legendrian rational cubics through the points $(0,0,0,1),(1,1,1,1),(-1,1,-1,1)$ is 
\begin{equation}
\label{eq_allcubics}
(3t-t^3, 2t^2+2\mu(t-t^3), 2t^3,1+t^2-2\mu(t-t^3)).
\end{equation}
The surface swept by all these cubics is given by $F=0$ where $$F(x,y,z,w)=2x^3+21x^2z-27y^2z-54yzw-27zw^2+60xz^2+25z^3.$$ Such a surface intersects a generic line in three points, this gives another proof of Corollary~\ref{cor_answer}.

\end{corollary}

\section{Proof of Theorem~\ref{th_contact}}

Each rational cubic curve has a parametrization of the form 
\begin{align*}
(a_0+a_1t+a_2t^2+a_3t^3,b_0+b_1t+b_2t^2+b_3t^3,c_0+c_1t+c_2t^2+c_3t^3,d_0+d_1t+d_2t^2+d_3t^3).
\end{align*}
We supposed that our cubic passes through points $(0,0,0,1),(1,1,1,1),(-1,1,-1,1)$ at $t=0,1,-1$ respectively. Substituting $t=0$ in the parametrization, we obtain $a_0=b_0=c_0=0,d_0=1$. Substitutions $t=\pm 1$ give us  

\begin{align*}
a_1+a_2+a_3&=b_1+b_2+b_3=c_1+c_2+c_3=1+d_1+d_2+d_3, \\
a_1-a_2+a_3&=-b_1+b_2-b_3=c_1-c_2+c_3=1-d_1+d_2-d_3.
\end{align*}

Therefore, $a_2=b_1+b_3=c_2=d_1+d_3, a_1+a_3=b_2=c_1+c_3=1+d_2$.

Substituting indeterminates with bigger indices as functions of the indeterminates with
smaller indices we obtain that our curve is parametrized by
\begin{align*}
(a_1t+a_2t^2+&(b_2-a_1)t^3,b_1t+b_2t^2+(a_2-b_1)t^3,\\
&c_1t+a_2t^2+(b_2-c_1)t^3,1+d_1t+(b_2-1)t^2+(a_2-d_1)t^3).
\end{align*}

Evaluating the form $3ydx-3xdy+wdz-zdw$ on the curve we obtain

\begin{align*}
3(b_1t+b_2t^2+&(a_2-b_1)t^3)(a_1+2a_2t+3(b_2-a_1)t^2)\\
&-3(a_1t+a_2t^2+(b_2-a_1)t^3)(b_1+2b_2t+3(a_2-b_1)t^2)\\
&+(1+d_1t+(b_2-1)t^2+(a_2-d_1)t^3)(c_1+2a_2t+3(b_2-c_1)t^2)\\
&-(c_1t+a_2t^2+(b_2-c_1)t^3)(d_1+2(b_2-1)t+3(a_2-d_1)t^2) =0.
\end{align*}

The coefficient before $t^0$ should be equal to $0$, so $c_1=0$. The parametrization rewrites as
\begin{align*}
3(b_1t&+b_2t^2+(a_2-b_1)t^3)(a_1+2a_2t+3(b_2-a_1)t^2)-3(a_1t+a_2t^2+(b_2-a_1)t^3)(b_1+2b_2t+3(a_2-b_1)t^2)\\
&+(1+d_1t+(b_2-1)t^2+(a_2-d_1)t^3)(2a_2t+3b_2t^2) -(a_2t^2+b_2t^3)(d_1+2(b_2-1)t+3(a_2-d_1)t^2) = 0.
\end{align*}

Coefficient before $t^1$ equals $2a_2$, so $a_2=0$.

\begin{align*}
3(b_1t+b_2t^2-&b_1t^3)(a_1+3(b_2-a_1)t^2)-3(a_1t+(b_2-a_1)t^3)(b_1+2b_2t-3b_1t^2)\\
&+(1+d_1t+(b_2-1)t^2-d_1t^3)(3b_2t^2) -(b_2t^3)(d_1+2(b_2-1)t-3d_1t^2) = \\
3(b_1t+b_2t^2-&b_1t^3)(a_1+3(b_2-a_1)t^2)-3(a_1t+(b_2-a_1)t^3)(b_1+2b_2t-3b_1t^2)\\
&+b_2t^2(3+3d_1t+3(b_2-1)t^2-3d_1t^3 -d_1t-2(b_2-1)t^2+3d_1t^3)=\\
3(b_1t-b_1t^3)(a_1+3&(b_2-a_1)t^2)-3(a_1t+(b_2-a_1)t^3)(b_1-3b_1t^2)\\
&+b_2t^2(3(b_2-a_1)t^2-3a_1+3+2d_1t+(b_2-1)t^2) =\\
b_1t^3(-3a_1-&9(b_2-a_1)t^2+9(b_2-a_1)+9a_1-3(b_2-a_1)+9(b_2-a_1)t^2)+b_1t(3a_1-3a_1)\\
&+b_2t^2(3(b_2-a_1)t^2-3a_1+3+2d_1t+(b_2-1)t^2)=\\
6b_1b_2&t^3+b_2t^2(3(b_2-a_1)t^2-3a_1+3+2d_1t+(b_2-1)t^2)=\\
b_2t^2(6b_1t&+3(b_2-a_1)t^2-3a_1+3+2d_1t+(b_2-1)t^2) =\\
b_2t^2(t(6&b_1+2d_1)+t^2(4b_2-3a_1-1)-3a_1+3) =0\\
\end{align*}

Therefore, either $b_2=0$ or $a_1=1, b_2=1,d_1=-3b_1$.

In the first case the curve is going to be as follows:
$$(a_1t-a_1t^3,b_1t-b_1t^3,0,1+d_1t-t^2-d_1t^3) = (a_1t,b_1t,0,1+d_1t).$$

what is not really a cubic, but in the second case we have

$$(t,b_1t+t^2-b_1t^3,t^3,1-3b_1t+3b_1t^3) = (t,t^2+\mu(t-t^3),t^3, 1-3\mu(t-t^3)).$$

As it was predicted by Macaulay2, we have obtained a linear family of cubics.

\begin{remark}
One can look at what happens in the limiting case $\mu=\infty$. The
family of curves converges (if we look at the parametrizations) to a point $(0,-1/3,0)$. On the other
hand their tangent vectors at $t=0,1,-1$ converge to
$(0,1,0),(-3,-4,-3),(3,-4,3)$ respectively. Then, contact lines from
$(0,0,0,1),(1,1,1,1),(-1,1,-1,1)$ with these tangent vectors all intersect
in $(0,-1/3,0)$. So, the family $l(t,\mu)$ converges to these three lines as $\mu\to\infty$, these three lines with the embedded point $(0,-1/3,0)$ is a point on the boundary of the Hilbert scheme of rational cubics in $\CC P^3$ (see \cite{cubicsComp} for more details about the compactification of the space of rational cubics).
\end{remark}

\begin{remark}
\label{pr_legendrian}
Is it true for higher degrees? A hypothesis: there always exist at least
$d$ legendrian rational curves of degree $d$ passing through $d$ generic points and a line. An heuristic argument is as follows. We take the one dimensional family (because $\mathrm{Stab^3_\mu}$ acts on these curves) of the degree $d$ legendrian curves through $d$ points which all belong to a given plane $L$, and write the equation of the surface that they sweep. Then we intersect this surface with $L$. We obtain a collection of $d$ lines in the intersection, therefore the degree of the surface is at least $d$, therefore there is al least $d$ legendrian curves through $d$ generic points and one generic line. Also, this approach by perturbation of degenerate families might work for any genus, as long as the set of the curves is not empty.
\end{remark}

%\q{find a one-dim family of three points such that these families
%  cover all $\CC P^3$.}

%\section{Cubic surface containing the family of legendrian cubics}
%
%The equation of the surface containing all the cubics in the previous section is $$F(x,y,z,w)=2x^3+21x^2z-27y^2z-54yzw-27zw^2+60xz^2+25z^3.$$ Such a surface intersects a generic line in three points, this gives another proof of Corollary~\ref{cor_answer}.
%
%We found the family of legendrian cubics through $(0,0,0,1),(1,1,1,1),(-1,1,-1,1)$ and tangent to $\omega_1=3ydx-3xdy+wdz-zdw$. Therefore the family $(3t,t^2+\mu(t-t^3),t^3, 1-3\mu(t-t^3))$ is tangent to $\omega=ydx-xdy+wdz-zdw$ but passes through $(0,0,0,1),(3,1,1,1),(-3,1,-1,1)$. Therefore we apply the contactomorphism 
%
%$
%\psi=
%\begin{pmatrix}
%1/2 & 0 & -1/2 & 0 \\
%0 & 1 & 0 & 0 \\
%0 & 0 & 1 & 0 \\
%0 & 1/2 & 0 & 1/2 \\
%\end{pmatrix},$
%
%which brings these three points to the standard ones.
%
%After the action of $\psi$, the parametrization of the family of legendrian rational cubics through the points $(0,0,0,1),(1,1,1,1),(-1,1,-1,1)$ is 
%\begin{equation}
%\label{eq_allcubics}
%(3t-t^3, 2t^2+2\mu(t-t^3), 2t^3,1+t^2-2\mu(t-t^3)).
%\end{equation}
% In the following computation we choose the affine coordinates $(x/z,y/z,1,w/z)$. The following code produce the equation $F=0$ mentioned in the beginning of this section.

%\q{think about what type of compatification should we consider and how
%to fibre the  whole space on cubics, probably blow up something.}
%\section*{References}
\bibliographystyle{plain}

\bibliography{../../bibliography}

\end{document}